\documentclass[10pt]{amsart}
\usepackage{a4wide}
\usepackage{amsmath}
\usepackage{amsfonts}
\usepackage{amssymb}
\usepackage{graphicx}
\usepackage{amsthm,graphicx,color,yfonts}
\usepackage{pdfsync}
\usepackage{epstopdf}%
\usepackage{pifont}
\usepackage{bbm}

\usepackage[colorlinks=true]{hyperref}
\hypersetup{linkcolor=red,citecolor=blue,filecolor=dullmagenta,urlcolor=magenta} 

\newcommand{\R} {\mathbb R}
\newcommand{\cuad}{{\sqcap\kern-.68em\sqcup}}

\newcommand{\ve}{\varepsilon}

\newcommand{\be}{\begin{equation}}
\newcommand{\ee}{\end{equation}}

\definecolor{darkgreen}{rgb}{0.2,0.7,0.1}

\newcommand{\re}{\operatorname{Re}}
\newcommand{\ima}{\operatorname{Im}}

\newcommand{\Com}{\mathbb{C}}

\newcommand{\al}{\alpha}
\newcommand{\bt}{\beta}
\newcommand{\ga}{\gamma}

\newcommand{\supp}{\operatorname{supp}}

\def\bm{\left( \begin{array}{cc}}
\def\endm{\end{array}\right)}

\providecommand{\norm}[1]{\left\| #1 \right\|}

\newcommand{\ba}{\begin{equation*}}
\newcommand{\ea}{\begin{equation*}}
\newcommand{\bea}{\begin{eqnarray}}
\newcommand{\eea}{\end{eqnarray}}
\newcommand{\bee}{\begin{eqnarray*}}
\newcommand{\eee}{\end{eqnarray*}}
\newcommand{\ben}{\begin{enumerate}}
\newcommand{\een}{\end{enumerate}}

\numberwithin{equation}{section}

\newtheorem{thm}{Theorem}[section]
\newtheorem*{theorem*}{Theorem}

\newtheorem{cor}{Corollary}[section]
\newtheorem{lem}{Lemma}[section]
\newtheorem{defn}{Definition}[section]

\theoremstyle{remark}
\newtheorem{rem}{Remark}[section]

\title[Nonlocalized NLS solutions]{Instability in Nonlinear Schr\"odinger breathers}

\author{Claudio Mu\~noz}
\address{CNRS and Departamento de Ingenier\'{\i}a Matem\'atica and Centro
de Modelamiento Matem\'atico (UMI 2807 CNRS), Universidad de Chile, Casilla
170 Correo 3, Santiago, Chile.}
\email{claudio.munoz@math.u-psud.fr, cmunoz@dim.uchile.cl}
\thanks{C. M.  was partly funded by Chilean research grants FONDECYT  1150202, Fondo Basal CMM-Chile, and Millennium
Nucleus Center for Analysis of PDE NC130017}


\begin{document}

\begin{abstract}
We consider the \emph{focusing} Nonlinear Schr\"odinger equation posed on the one dimensional line, with nonzero background condition at spatial infinity, given by a homogeneous plane wave. For this problem of physical interest, we study the initial value problem for perturbations of the background wave in Sobolev spaces. It is well-known that the associated linear dynamics for this problem describes a phenomenon known in the literature as \emph{modulational instability}, also recently related to the emergence of \emph{rogue waves} in ocean dynamics. In qualitative terms, small perturbations of the background state increase its size exponentially in time. In this paper we show that, even if  there is no time decay for the linear dynamics due to the modulationally unstable regime, the equation is still locally well-posed in $H^s$, $s>\frac12$. We apply this result to give a rigorous proof of the unstable character of two well-known NLS solutions: the Peregrine and Kuznetsov-Ma  breathers.  
\end{abstract}

\maketitle

\section{Introduction and Main results}

\subsection{Setting of the problem} In this paper, we consider the \emph{focusing} nonlinear Schr\"odinger equation (NLS) in one dimension:
\be\label{NLS}
i\partial_t u + \partial_x^2 u +|u|^{p-1} u=0, \quad u=u(t,x)\in \Com, \quad t,x\in \R.
\ee
Along this paper we assume $p >1$. The initial value problem (IVP) for \eqref{NLS} when the initial datum $u(t=0)=u_0$ is in standard Sobolev spaces $H^s(\R)$ is by now well-understood, in particular in the case of subcritical nonlinearities ($p<5$), starting with the fundamental works by Ginibre and Velo \cite{GV}, Tsutsumi \cite{Tsutsumi} and Cazenave and Weissler \cite{CW}, which showed global well-posedness (GWP) for $p<5$ and local well-posedness (LWP) if $p>1$. For a detailed description of the literature, including all historic developments, see e.g. the monographs by Cazenave \cite{Cazenave} and Linares and Ponce \cite{LP}. The case $p=3$ (cubic NLS) is particularly important because the equation is completely integrable, as showed by Zakharov and Shabat \cite{Zakharov0}. Additionally, this equation appears as a model of propagation of light in nonlinear optical fibers (with different meanings for time and space variables), as well as in small-amplitude gravity waves on the surface of deep inviscid water.

\medskip

As a complement to the previous results, concerning localized solutions only, in this work we are interested in constructing solutions to \eqref{NLS} for which the \emph{modulational instability} phenomenon is present. Being more precise, let us assume for simplicity that $p=3$, although our results hold for any $p>1$, provided the regularity in $H^s$ is properly chosen  (see Remark \ref{4.10.3}). Recall the standard \emph{Stokes wave} 
\be\label{Stokes}
u(t,x) := e^{it},
\ee
a non-localized, homogeneous solution of \eqref{NLS}. A complete family of standing waves can be obtained by using the scaling, phase and Galilean invariances of \eqref{NLS}: 
\be\label{standing_wave}
u_{c,v,\ga}(t,x):= \sqrt{c} \, \exp \Big(  ict + \frac i2 xv  -\frac i4 v^2 t + i\ga \Big).
\ee
This wave is another solution to \eqref{NLS}, for any scaling $c>0$, velocity $v\in \R$, and phase $\ga \in \R$. However, since all these symmetries represent invariances of the equation, they will not be essential in our proofs, and we will assume $c=1$, $v=\ga=0$.

\medskip

We now break the symmetry of the problem. Consider \emph{localized} perturbations of \eqref{Stokes} in \eqref{NLS}, of the form
\be\label{deco}
u(t,x) = e^{it}(1+ w(t,x)), \quad w~ \hbox{ unknown}.
\ee
(This rupture is motivated by some exact solutions to \eqref{NLS} discussed in subsection \ref{Breathers} below.) Then \eqref{NLS} becomes a modified NLS equation with a zeroth order term, which is real-valued, and has the wrong sign:
\be\label{mNLS}
\begin{aligned}
i\partial_t w + \partial_x^2 w  & + 2 \re w =~ G[w] , \\
 G[w]:= & - ( |1+ w|^{2}-1) (1+w) + 2 \re w  =  O(|w|^2).
\end{aligned}
\ee
The associated linearized equation for \eqref{mNLS} is just\footnote{This equation is similar to the well-known linear Sch\"odinger $i\partial_t w + \partial_x^2 w=0$, but instead of dealing with the additional term $2 \re w$ only as a perturbative term, we will consider all linear terms as a whole for later purposes (not considered in this paper), in particular, \emph{long time existence and decay issues}, see e.g. \cite{GNT1,GNT2}.}
\be\label{LS}
i\partial_t w + \partial_x^2 w  + 2 \re w =0.
\ee
Written only in terms of $\phi =\re w$, we have the wave-like equation (compare with \cite{Goodman} in the periodic setting)
\be\label{phi_eqn}
\partial_t^2 \phi+\partial_x^4 \phi + 2 \partial_x^2 \phi =0.
\ee
This problem has some instability issues, as reveal a standard frequency analysis: looking for a formal standing wave $\phi=e^{i(kx-\omega t)}$ solution to \eqref{phi_eqn}, one has
\[
\omega(k) = \pm |k| \sqrt{k^2 -2},
\]
which reveals that for small wave numbers ($|k|<\sqrt{2}$) the linear equation behaves in an ``elliptic'' fashion, and exponentially (in time) growing modes are present from small perturbations of the vacuum solution. A completely similar conclusion is obtained working in the Fourier variable, as we will see below (see Section \ref{2}). This singular behavior is not present if now the equation is defocusing, that is \eqref{NLS} with nonlinearity $-|u|^{2} u$.\footnote{Another model corresponds to the Gross-Pitaevskii equation: $i\partial_t u + \partial_x^2 u +u(1-|u|^2)=0$, for which the Stokes wave is modulationally stable.}
\medskip

\subsection{Well-posedness} The simple phenomenon exposed above is part of an intensely studied effect known as \emph{modulational instability}, which -roughly speaking- says that small perturbations of the exact solution \eqref{Stokes} are unstable and grow quickly. This unstable growth leads to a nontrivial competition with the (focusing) nonlinearity, time at which the solution is apparently stabilized. Some examples revealing this behavior are the breathers for the integrable NLS equation (\eqref{NLS} with $p=3$ above), the most famous being the Peregrine breather (or soliton) \cite{Peregrine}
\be\label{Peregrine}
P(t,x):=  e^{it}\Big(1-\frac{4(1+2it)}{1+ 4t^2 +2x^2}\Big).
\ee
This exact solution is a space-time localized, nontrivial perturbation of the Stokes wave, which appears and disappears from nowhere \cite{Akhmediev2}. Some interesting connections have been made between the Peregrine soliton \eqref{Peregrine} and the intensely studied subject of \emph{rogue waves} in ocean \cite{Zakharov,Shrira,Akhmediev2,Kliber0}. Very recently, Biondini and Mantzavinos \cite{BM} showed (see also \cite{BM1}), using inverse scattering techniques, the existence and long-time behavior of a global solution to \eqref{mNLS} in the {integrable} case $(p=3)$, but under certain exponential decay assumptions at infinity, and a \emph{no-soliton} spectral condition (which, as far as we understand, does not define an open subset of the space of initial data). Motivated by this fundamental result, we asked ourselves whether or not a suitable notion of solution for \eqref{mNLS} exists in standard Sobolev spaces, for which one can study solutions like \eqref{Peregrine}. Consequently, in this paper we show that, despite the fact that \eqref{LS} is not ``suitable well-posed'' as usually well-known dispersive models are, \eqref{mNLS} has local-in-time strong (and continuous in space) solutions in Sobolev spaces of fractional order.

\begin{thm}\label{MT}
The modulationally unstable equation \eqref{mNLS} is locally well-posed for any initial data in $H^s$, $s>\frac12$.
\end{thm}

See Theorem \ref{MT1} for a detailed statement of Theorem \ref{MT}. Compared with the results in \cite{BM}, we say less about the long-time dynamics, but we define a local-in-time flow for $w$ on a open set of the origin in $H^s$, with minimal assumptions. Note that $P$ in \eqref{Peregrine} is always well-defined, and has essentially no loss of regularity, confirming in some sense the intuition and the conclusions in Theorem \ref{MT}. Also, note that using the symmetries of the equation, we have LWP for any solution of \eqref{NLS} of the form
\[
u(t,x) = u_{c,v,\ga}(t,x) +w(t,x), \quad w\in H^s, \quad s>\frac12,
\]
with $u_{c,v,\ga}$ defined in \eqref{standing_wave}.
\medskip

The main feature in the proof of Theorem \ref{MT} is the fact that, if we work in Sobolev spaces, in principle \emph{there is no $L^1-L^\infty$ decay estimates for the linear dynamics}. Moreover, one has exponential growth in time of the $L^2$ norm, and therefore no suitable Strichartz estimates seems to be available, unless one cuts off some bad frequencies. Consequently, Theorem \ref{MT} is based in the fact that we work in dimension one, and that for $s>\frac12$, we have the inclusion $H^s \hookrightarrow  L^\infty$. 

\medskip

Let us discuss in more detail the weak form of instability present in equation \eqref{mNLS}, when working in Sobolev spaces. In order to understand why \eqref{mNLS} is still well-posed, consider the {\it backward heat equation} in $\R^d$ 
\be\label{Backwards}
\partial_t u = -\Delta u, \quad u=u(t,x) \in \R, \quad t\geq 0,
\ee
with initial datum $u(t=0)=u_0 \in \mathcal S(\R^d)$, the standard Schwartz class. A simple Fourier analysis reveals that the solution is given by the representation
\[
u (t,x) = const.\int e^{ix\cdot \xi } e^{|\xi|^2 t} \hat u_0(\xi)d\xi.   
\] 
Even if $\hat u_0 \in \mathcal S(\R^d)$, for $t>0$ sufficiently large there are initial data $u_0$ for which   $ e^{|\xi|^2 t} \hat u_0(\xi)$ becomes unbounded in $\xi$, and the solution ceases to exist in the original Schwartz class. Therefore, \eqref{Backwards} is ill-posed in the Hadamard's sense \cite{Hadamard}. However, if now $\hat u_0$ has compact support in $\xi$, one has the uniform (in space) estimate
\be\label{Expo}
\Big| \int e^{ix\cdot \xi } e^{|\xi|^2 t} \hat u_0(\xi)d\xi \Big|~ \lesssim  e^{M^2 t}, \quad \supp \hat u_0 \subseteq B(0,M).
\ee
Therefore one has exponential growth in time, but the solution is fortunately well-defined for all time. A similar conclusion can be reached in the case of \eqref{phi_eqn}: the ``ill-posed'' regime in the Fourier space corresponds to the bounded region $|\xi|\leq \sqrt{2}$, while for $|\xi|>\sqrt{2}$ the equation is naturally dispersive, with standard decay estimates. As in \eqref{Expo}, the price to pay when dealing with \eqref{phi_eqn} is an exponential growth in time at the linear level (even for energy estimates), which reduces the well-posedness result to the only \emph{general} regime where one can close the fixed point iteration of global estimates\footnote{Possibly, one can also prove existence of globally defined weak solutions in $H^1(\R)$, as in Leray \cite{Leray} (Navier-Stokes), and Kato \cite{Kato} (modified KdV).}: the subcritical Sobolev setting, in this case $H^s$, $s>\frac12.$ 

\begin{rem}
Note that \eqref{mNLS} differs from another well-known, ill-posed fluid equation, the ``bad'' Boussinesq model \cite[p. 66]{KL}:
\[
\partial_t^2 u -  \partial_x^4u -\partial_x^2 u +  \partial_x^2(u^2)=0,
\]
for which the linear problem, in the Fourier space, reads (compare with \eqref{phi_eqn} and \eqref{AA})
\[
\partial_t^2 \hat u + \xi^2(1-\xi^2)\hat u =0.
\]
Here, the set of bad frequencies is unbounded: $\{|\xi|\geq 1\}$. 
\end{rem}

\begin{rem}
A related local well-posedness result as in Theorem \ref{MT} has been established for \eqref{NLS} in the \emph{Zhidkov space}, or energy space
\be\label{Zhidkov_space}
\mathcal E := \{ u\in L^\infty(\R) ~ : ~ \partial_x u \in L^2(\R), ~ |u|^2-1 \in L^2(\R) \},
\ee
by following the ideas by Zhidkov \cite{Zhidkov}, see also Gallo \cite{Gallo,Gallo2}, and G\'erard \cite{Gerard}, valid for any subcritical-critical dimension. These papers are devoted to the defocusing case (the Gross-Pitaevskii equation), where they also obtain global well-posedness thanks to the nonnegativity of the energy. When $s=1$ in Theorem \ref{MT}, the space $\mathcal E$ contains the solution space $e^{it}(1 + H^1(\R))$, but for $s\in (\frac12,1)$ there is no evident relationship between these spaces. Also, scattering results for Gross-Pitaevskii equations are proved by Gustafson, Nakanishi and Tsai in \cite{GNT1,GNT2}. Some of the linear dynamics in Section \ref{2} is related to the one in these works, the main difference being the modulationally unstable character of the evolution problem considered in this work. 
\end{rem}

\begin{rem}
After this work was completed, we learned about an alternative explanation to the rogue wave phenomena, given by the notion of \emph{dispersive blow-up} (see Bona and Saut \cite{Bona_Saut}), in which the solution, even if it is well-defined in $L^2$-based Sobolev spaces, has $L^\infty$ norm becoming unbounded in finite time. This argument works for linear and nonlinear equations as well. Therefore, a necessary condition for getting dispersive blow-up is an $H^s$ regularity where $s<\frac d2$, $d$ being the dimension of space, which does not fit our LWP assumptions. For further improvements and applications of these ideas to other dispersive models, see the work by Bona et al. \cite{BPSS}.  
\end{rem}

Two interesting (mathematical) questions left open in this work are the following: global existence vs. blow-up (see Remark \ref{GE} for more details), and well-posedness and ill-posedness of the flow map for lower regularities. Note that the ill-posedness method developed by Kenig-Ponce-Vega \cite{KPV} does not apply since scaling and Galilean transformations require infinite $H^s$ energy. 

\medskip

We will apply Theorem \ref{MT} to show some additional results, in particular, stability issues for particular NLS soliton solutions appearing when $p=3$. 

\subsection{More about modulational instability}\label{Breathers} In addition to the Stokes wave \eqref{Stokes}, it is also known that the Peregrine soliton has a modulational instability property \cite{Khawaja}, and it is numerically unstable under small perturbations, see the work by Klein and Haragus \cite{KH}. Since the equation is locally well-posed and continuous in time, it is possible to define a notion of orbital stability for the Peregrine breather, but also for more general solutions. 

\medskip

Fix $s>\frac12$, and $t_0\in \R$. We say that a particular globally defined solution $U=e^{it}(1+W)$ of \eqref{NLS} is \emph{orbitally stable} in $H^s$ if there are constants $C_0,\ve_0>0$ such that, for any $0<\ve<\ve_0$,  
\be\label{Stability}
\begin{aligned}
\| w_0 - & W(t_0)\|_{H^s} <  \ve\\
& \Downarrow  \\
\exists ~ x_0(t)\in \R ~  \hbox{such that }~ &  \sup_{t\in \R}  \|w(t) -W(t,x-x_0(t)) \|_{H^s} <C_0\, \ve.
\end{aligned}
\ee

\medskip
\noindent
Here $w(t)$ is the solution to the IVP \eqref{mNLS} with initial datum $w(t_0)=w_0$, constructed in Theorem \ref{MT}, and $x_0(t)$ can be assumed continuous because the IVP is well-posed in a continuous-in-time Sobolev space.\footnote{Note that no phase correction is needed in \eqref{Stability}: equation \eqref{mNLS} is no longer $U(1)$ invariant, and any phase perturbation of a modulationally unstable solution $u(t)$ in \eqref{NLS}, of the form $u(t) e^{i\ga}$, $\ga\in \R$, requires an infinite amount of energy. The same applies for Galilean transformations.} If \eqref{Stability} is not satisfied, we will say that $U$ is {\bf unstable}. Note additionally that condition \eqref{Stability} requires $w$  globally defined, otherwise $U$ is trivially unstable, since $U$ is globally defined.

\medskip

Recall that NLS solitons on a zero background satisfy \eqref{Stability} (with an additional phase correction) for $s=1$ and all $L^2$ subcritical nonlinearities, see e.g. \cite{CL,GSS,Weinstein}.

\medskip
Let us come back to the Peregrine breather \eqref{Peregrine}.  The following result, which is the main result of this paper, quantifies the instability of the Peregrine breather.

\begin{thm}\label{Insta_Per}
The Peregrine breather \eqref{Peregrine} is unstable under small $H^s$ perturbations, $s>\frac12$.
\end{thm}

This result is in contrast with other positive results involving breather solutions \cite{AM,AM1,AM2,Munoz}. In those cases, the involved equations (mKdV, sine-Gordon) were globally well-posed in the energy space (and even in smaller subspaces), with uniform in time bounds. Several physical and computational studies on the Peregrine breather can be found in \cite{Dysthe,Brunetti} and references therein. The proof of Theorem \ref{Insta_Per} will be a direct application of the notion of modulational instability together with an asymptotic stability property. It will be also clear from the proof that Peregrine breathers cannot be asymptotically stable.

\medskip

There is an additional oscillatory mode for \eqref{NLS}-\eqref{deco}. This mode corresponds to a perturbation of the Stokes wave that is localized in space, and periodic in time. Assume $a>\frac12$. The Kuznetsov-Ma (KM) breather \cite{Kuznetsov,Ma} is given by the compact expression \cite{Akhmediev}
\be\label{KM}
\begin{aligned}
B(t,x): = & ~ e^{it}\Bigg[  1- \sqrt{2}\beta \frac{(\beta^2 \cos(\al t)  + i\al \sin(\al t)) }{ \al \cosh(\beta x) - \sqrt{2} \beta \cos(\al t)}\Bigg], \\
 \al := &~ (8a(2a-1))^{1/2},\quad \beta := (2(2a-1))^{1/2}. \\
\end{aligned}
\ee
In the formal limit $a \downarrow \frac12$ one recovers the Peregrine breather.\footnote{Note that $  \frac{\al}{\sqrt{2}\beta} = \sqrt{2a} >1
$, therefore $B$ in \eqref{KM} is always well-defined.} Note that $B$ is a Schwartz perturbation of the Stokes wave, and therefore a smooth classical solution of \eqref{NLS}.  It has been also observed in optical fibre experiments, see Kliber et al. \cite{Kibler} and references therein for a complete background on the mathematical problem and its physical applications. However, using a similar argument as in the proof of Theorem \ref{Insta_Per}, we will show that this mode cannot be stable, at least in our current definition of stability.

\begin{thm}\label{Insta_KM}
All Kuznetsov-Ma breathers are unstable under small $H^{s}$, $s>\frac12$ perturbations.
\end{thm}

The (formally) unstable character of Peregrine and Kuznetsov-Ma breathers was well-known in the physical and fluids literature (they arise from modulational instability), therefore the rigorous conclusions in Theorems \ref{Insta_Per} and \ref{Insta_KM} are somehow not surprising. However, in several water tanks or optic fiber experiments, researchers are able to reproduce these waves \cite{Brunetti,Kliber0,Kibler}, if e.g. the initial setting or configuration is close to the exact theoretical solution. An interesting physical and mathematical question should be to generalize the notion of stability in \eqref{Stability} in order to encompass the physical occurrence of these phenomena.

\medskip

We finish this section by recalling that NLS \eqref{NLS} possesses a third oscillatory mode, the Akhmediev breather \cite{Akhmediev}
\[
\begin{aligned}
A(t,x):=& ~ e^{it}\Bigg[  1+ \frac{\alpha^2 \cosh( \beta t) +i\beta \sinh(\beta t) }{ \sqrt{2a} \cos(\alpha x) -\cosh(\beta t)}\Bigg],\\
 \beta=&~ (8a(1-2a))^{1/2}, \quad \alpha=(2(1-2a))^{1/2},\quad a<\frac12,
\end{aligned}
\]  
In the limiting case $a \uparrow \frac12$ one can recover the Peregrine soliton \eqref{Peregrine}. Unlike Peregrine and Kuznetsov-Ma breathers, this solution is  periodic in space, and localized in time. The three breathers presented along this work can be seen more clearly in \cite{Youtube}. It would be interesting to study the stability of this solution, as the previous modes. The Cauchy problem in the periodic setting should be treated as is done in this work. We conjecture that this solution should be unstable under suitable periodic perturbations, for similar reasons to the previously mentioned.

\medskip

\noindent
{\bf Organization of this paper.} This paper is organized as follows. In Section \ref{2} we describe the linear dynamics in terms of its Fourier estimates. Section \ref{3} is devoted to the proof of Theorem \ref{MT}, while Section \ref{4} deals with the instability results in Theorems \ref{Insta_Per} and \ref{Insta_KM}.

\medskip

\noindent
{\bf Acknowledgments.} C.M. is indebted to M.A. Alejo, D. Eeltink and C. Sparber for many suggestions and clarifying discussions about modulational instability, NLS breathers, and dispersive blow-up.

\bigskip

\section{Fourier description of the linear dynamics}\label{2}

The purpose of this section is to obtain Fourier estimates for solutions $w$ to the linear problem
\be\label{Linear}
i\partial_t w +\partial_x^2w +2\re w =G, \qquad w(t=0)=w_0.
\ee
Along this section, and in the remaining part of this paper, we will use the following notation:
\be\label{deco_u0}
w = \re w + i\ima  w =:\phi + i \varphi, \quad  w_0 =: \phi_0+ i \varphi_0,
\ee
and
\be\label{deco_F}
G= \re G+ i \ima G= : f+ig,
\ee
where $\phi,\varphi, \phi_0, \varphi_0,f,g$ are real-valued functions. Recall that $\hat f(\xi) = \mathcal F[f](\xi)$ represents the Fourier transform of $f=f(x)$. We will always assume $t\geq 0$, but all our estimates are valid for $t<0$, if we use absolute value when necessary.

\begin{lem}\label{Fourier_Linear}
Assume that $w=\phi + i \varphi$ solves \eqref{Linear} with \eqref{deco_u0} and \eqref{deco_F} satisfied. Then we have the representation
\be\label{Duhamel}
 \phi = \Phi [f,g ; \phi_0,\varphi_0], \quad \varphi  =\Psi[f,g ; \phi_0,\varphi_0],
\ee
where
\ben
\item for frequencies $|\xi| \leq \sqrt{2}$, 
\be\label{phi_low}
\begin{aligned}
\mathcal F[ \Phi [f,g ; \phi_0,\varphi_0]] (t,\xi) =&~ \cosh(|\xi| \sqrt{2-\xi^2}~ t) \hat\phi_0(\xi) ~ + \frac{ \sinh(|\xi| \sqrt{2-\xi^2}~ t)}{|\xi|\sqrt{2-\xi^2}} \xi^2 \hat\varphi_0(\xi) \\
& ~ + \int_0^t \cosh(|\xi| \sqrt{2-\xi^2} (t-\sigma))\hat g(\sigma,\xi) d\sigma \\
& ~  -\int_0^t  \frac{ \sinh(|\xi| \sqrt{2-\xi^2} (t-\sigma))}{|\xi|\sqrt{2-\xi^2}} \xi^2 \hat f(\sigma,\xi)d\sigma,
\end{aligned}
\ee
and
\be\label{varphi_low}
\begin{aligned}
\mathcal F[\Psi[f,g ; \phi_0,\varphi_0]] (t,\xi) =&~ \cosh(|\xi|\sqrt{2-\xi^2} ~ t) \hat\varphi_0(\xi)  ~+ \frac{\sinh(|\xi|\sqrt{2-\xi^2} ~ t) }{|\xi| \sqrt{2-\xi^2}} (2- \xi^2) \hat\phi_0(\xi)  \\
 & ~- \int_0^t  \cosh(|\xi|\sqrt{2-\xi^2} (t-\sigma)) \hat f(\sigma,\xi) d\sigma\\
 &~ + \int_0^t  \frac{\sinh(|\xi|\sqrt{2-\xi^2} (t-\sigma)) }{|\xi| \sqrt{2-\xi^2}}(2-  \xi^2) \hat g(\sigma,\xi)d\sigma.
\end{aligned}
\ee
\medskip

\item For frequencies $|\xi| > \sqrt{2}$, 
\be\label{phi_high}
\begin{aligned}
\mathcal F[ \Phi [f,g ; \phi_0,\varphi_0]] (t,\xi)=&~ \cos(|\xi| \sqrt{\xi^2-2}~ t) \hat\phi_0(\xi) ~ + \frac{ \sin(|\xi| \sqrt{\xi^2-2}~ t)}{|\xi|\sqrt{\xi^2-2}}  \xi^2 \hat\varphi_0 (\xi) \\
& ~ + \int_0^t \cos(|\xi| \sqrt{\xi^2-2} (t-\sigma))\hat g(\sigma,\xi) d\sigma \\
& ~  -\int_0^t  \frac{ \sin(|\xi| \sqrt{\xi^2-2} (t-\sigma))}{|\xi|\sqrt{\xi^2-2}} \xi^2 \hat f(\sigma,\xi)d\sigma,
\end{aligned}
\ee
and
\be\label{varphi_high}
\begin{aligned}
\mathcal F[ \Psi [f,g ; \phi_0,\varphi_0]] (t,\xi) =&~ \cos(|\xi|\sqrt{\xi^2-2} ~ t) \hat\varphi_0(\xi) ~+ \frac{\sin(|\xi|\sqrt{\xi^2-2} ~ t) }{|\xi| \sqrt{\xi^2-2}} (2- \xi^2) \hat\phi_0(\xi) \\
 & ~- \int_0^t  \cos(|\xi|\sqrt{\xi^2-2} (t-\sigma)) \hat f(\sigma,\xi) d\sigma\\
 &~ + \int_0^t  \frac{\sin(|\xi|\sqrt{\xi^2-2} (t-\sigma)) }{|\xi| \sqrt{\xi^2-2}} (2-\xi^2)\hat g(\sigma,\xi)  d\sigma.
\end{aligned}
\ee

\een

\end{lem}
\begin{proof}
For the proof of this result, see Appendix \ref{A}.
\end{proof}

\begin{rem}
Note that the ``bad case'' $|\xi| \leq \sqrt{2}$ in \eqref{phi_low}-\eqref{varphi_low} contains an exponential growth (in time) for frequencies $|\xi|\sim 1$. This sort of ``weak'' ill-posedness behavior seems being impossible to avoid, but fortunately it is only present in a compact set of frequencies. Note also that we could have had initial data with support contained in the region $|\xi|>\sqrt{2}$ of the Fourier space, however, such a property is apparently not preserved by the nonlinear dynamics. 
\end{rem}

\subsection{Energy estimates for low frequencies}

Now we prove some simple energy estimates for the unstable case, namely $|\xi|\leq \sqrt{2}$. We first deal with the ``homogeneous'' estimates.

\begin{lem} Let $s\geq 0$ be a real-valued number, and $t\geq 0$. Consider the symbols introduced in \eqref{phi_low}-\eqref{varphi_low}. Then we have
\be\label{Estimate_11}
\norm{|\xi|^s \cosh(|\xi| \sqrt{2-\xi^2}~ t) \hat\phi_0(\xi)}_{L^2(|\xi|\leq \sqrt{2})} \lesssim_s \cosh t \norm{\hat\phi_0(\xi)}_{L^2(|\xi|\leq \sqrt{2})},
\ee
\be\label{Estimate_12}
\begin{aligned}
& \norm{ |\xi|^s \frac{ \sinh(|\xi| \sqrt{2-\xi^2} ~t)}{|\xi|\sqrt{2-\xi^2}} \xi^2 \hat\varphi_0 (\xi)}_{L^2(|\xi|\leq \sqrt{2})}  \lesssim_s  \sinh t \norm{\hat\varphi_0(\xi)}_{L^2(|\xi|\leq \sqrt{2})},
\end{aligned}
\ee
and finally

\be\label{Estimate_13}
\begin{aligned}
& \norm{ |\xi|^s\frac{\sinh(|\xi|\sqrt{2-\xi^2} ~ t) }{|\xi| \sqrt{2-\xi^2}} (2- \xi^2) \hat\phi_0(\xi)}_{L^2(|\xi|\leq \sqrt{2})} ~ \lesssim_s \sinh t \norm{\hat\varphi_0(\xi)}_{L^2(|\xi|\leq \sqrt{2})} .
\end{aligned}
\ee
\end{lem}

\begin{proof}
Let us prove \eqref{Estimate_11}. We have
\[
\begin{aligned}
& \int_{|\xi|\leq \sqrt{2}} |\xi|^{2s}\cosh^2(|\xi| \sqrt{2-\xi^2}~ t) |\hat\phi_0|^2(\xi)d\xi \\
& \qquad \qquad \leq ~ 2^s\sup_{|\xi|\leq \sqrt{2}} \cosh^2(|\xi| \sqrt{2-\xi^2}~ t) \norm{\hat\phi_0(\xi)}_{L^2(|\xi|\leq \sqrt{2})} \\
& \qquad \qquad \lesssim_s ~ \cosh^2 t \norm{\hat\phi_0(\xi)}_{L^2(|\xi|\leq \sqrt{2})}.
\end{aligned}
\]
Proof of \eqref{Estimate_12}. We have the standard inequality ($t\geq 0$)
\be\label{ineq_1}
\sup_{|\xi|\leq \sqrt{2} }\frac{\sinh(|\xi| \sqrt{2-\xi^2} ~t)}{|\xi|\sqrt{2-\xi^2} }\leq \sinh t.
\ee
For a proof of this result, see Appendix \ref{B}. From this estimate \eqref{Estimate_12} follows immediately.  The proof of \eqref{Estimate_13} is similar.
\end{proof}

Now we consider the inhomogeneous estimates.

\begin{lem}
 Consider again $s\geq 0$ and the symbols introduced in \eqref{phi_low}-\eqref{varphi_low}. Then we have
\be\label{Estimate_21}
\begin{aligned}
& \norm{|\xi|^s \int_0^t \cosh(|\xi| \sqrt{2-\xi^2} (t-\sigma))\hat g(\sigma,\xi) d\sigma }_{L^2(|\xi|\leq \sqrt{2})} \\
&\qquad\qquad  ~ \lesssim_s  \int_0^t    \cosh(t-\sigma) \norm{ \hat g(\sigma,\xi) }_{L^2(|\xi|\leq \sqrt{2})} d\sigma,
\end{aligned}
\ee
\be\label{Estimate_22}
\begin{aligned}
& \norm{|\xi|^s\int_0^t  \frac{ \sinh(|\xi| \sqrt{2-\xi^2} (t-\sigma))}{|\xi|\sqrt{2-\xi^2}} \xi^2 \hat f(\sigma,\xi)d\sigma}_{L^2(|\xi|\leq \sqrt{2})} \\
&\qquad\qquad  ~ \lesssim_s \int_0^t \sinh (t-\sigma)\norm{\hat f(\sigma,\xi)}_{L^2(|\xi|\leq \sqrt{2})}d\sigma,
\end{aligned}
\ee
and
\be\label{Estimate_23}
\begin{aligned}
& \norm{|\xi|^s \int_0^t  \frac{\sinh(|\xi|\sqrt{2-\xi^2} (t-\sigma)) }{|\xi| \sqrt{2-\xi^2}}(2-  \xi^2) \hat g(\sigma,\xi)d\sigma}_{L^2(|\xi|\leq \sqrt{2})}  \\
& \qquad\qquad  ~ \lesssim_s \int_0^t \sinh (t-\sigma)\norm{\hat g(\sigma,\xi)}_{L^2(|\xi|\leq \sqrt{2})}d\sigma.
\end{aligned}
\ee
\end{lem}
\begin{proof}
Proof of \eqref{Estimate_21}. We have
\[
\begin{aligned}
&  \norm{ |\xi|^s\int_0^t \cosh(|\xi| \sqrt{2-\xi^2} (t-\sigma))\hat g(\sigma,\xi) d\sigma }_{L^2(|\xi|\leq \sqrt{2})} \\
&~ \lesssim_s   \int_0^t    \norm{ \cosh(|\xi| \sqrt{2-\xi^2} (t-\sigma))\hat g(\sigma,\xi) }_{L^2(|\xi|\leq \sqrt{2})} d\sigma\\
&~ \lesssim_s  \int_0^t    \cosh(t-\sigma) \norm{ \hat g(\sigma,\xi) }_{L^2(|\xi|\leq \sqrt{2})} d\sigma.
\end{aligned}
\]
Proof of \eqref{Estimate_22}. Similar to the previous case, using \eqref{ineq_1} we have
\[
\begin{aligned}
& \norm{|\xi|^s \int_0^t  \frac{ \sinh(|\xi| \sqrt{2-\xi^2} (t-\sigma))}{|\xi|\sqrt{2-\xi^2}} \xi^2 \hat f(\sigma,\xi)d\sigma}_{L^2(|\xi|\leq \sqrt{2})} \\
&~\lesssim_s \int_0^t  \norm{  \frac{ \sinh(|\xi| \sqrt{2-\xi^2} (t-\sigma))}{|\xi|\sqrt{2-\xi^2}} \xi^2 \hat f(\sigma,\xi)}_{L^2(|\xi|\leq \sqrt{2})}d\sigma\\
&~\lesssim_s \int_0^t \sinh (t-\sigma)\norm{\hat f(\sigma,\xi)}_{L^2(|\xi|\leq \sqrt{2})}d\sigma.
\end{aligned}
\]
The proof of \eqref{Estimate_23} is completely similar. 
\end{proof}

From the previous estimates, the following results are immediate.
\begin{cor}\label{Fourier_low}
Assume $s\geq 0$, $t\geq 0$ and consider the Duhamel representation for $\phi$ and $\varphi$ given in \eqref{Duhamel}-\eqref{phi_low}-\eqref{varphi_low}. Then we have

\be\label{Hs_norm_phi_low}
\begin{aligned}
& \| |\xi|^s\mathcal F[ \Phi [f,g ; \phi_0,\varphi_0]] (t,\xi)\|_{L^2(|\xi|\leq \sqrt{2})} \\
 & \quad  \lesssim_s  \cosh t \norm{\hat\phi_0(\xi)}_{L^2(|\xi|\leq \sqrt{2})} + \sinh t  \norm{\hat\varphi_0(\xi)}_{L^2(|\xi|\leq \sqrt{2})} \\
& \qquad \qquad ~ + \int_0^t    \cosh(t-\sigma) \norm{ \hat g(\sigma,\xi) }_{L^2(|\xi|\leq \sqrt{2})} d\sigma ~ + \int_0^t \sinh (t-\sigma)\norm{\hat f(\sigma,\xi)}_{L^2(|\xi|\leq \sqrt{2})}d\sigma.
\end{aligned}
\ee
Similarly,

\be\label{Hs_norm_varphi_low}
\begin{aligned}
& \| |\xi|^s \mathcal F[ \Psi [f,g ; \phi_0,\varphi_0]] (t,\xi)\|_{L^2(|\xi|\leq \sqrt{2})} \\
 & \quad \lesssim_s ~ \cosh t  \norm{\hat\varphi_0(\xi)}_{L^2(|\xi|\leq \sqrt{2})}+  \sinh t \norm{\hat\varphi_0(\xi)}_{L^2(|\xi|\leq \sqrt{2})}  \\
& \qquad \qquad ~ +\int_0^t    \cosh(t-\sigma) \norm{ \hat f(\sigma,\xi) }_{L^2(|\xi|\leq \sqrt{2})} d\sigma   ~+\int_0^t \sinh (t-\sigma)\norm{\hat g(\sigma,\xi)}_{L^2(|\xi|\leq \sqrt{2})}d\sigma.
\end{aligned}
\ee

\end{cor}

\medskip

\subsection{Estimates for high frequencies} Now we deal with the ``stable'' case $|\xi|>\sqrt{2}$. Now we have oscillatory integrals with decay estimates, but since the previous unstable regime ($|\xi|<\sqrt{2}$) destroyed any possibility of an $L^1-L^\infty$ decay, we will only need to prove energy estimates for sufficiently large frequencies.

\begin{lem} Assume $s,t\geq 0$. Consider the representations \eqref{phi_high} and \eqref{varphi_high} for $\Phi$ and $\Psi$, respectively. Then we have
\be\label{Estimate_14}
\norm{ |\xi|^s\cos(|\xi| \sqrt{\xi^2-2} ~t) \hat\phi_0(\xi)}_{L^2(|\xi|> \sqrt{2})} ~ \leq ~ \norm{|\xi|^s \hat\phi_0(\xi)}_{L^2(|\xi|> \sqrt{2})},
\ee
\be\label{Estimate_15}
\begin{aligned}
& \norm{ |\xi|^s\frac{ \sin(|\xi| \sqrt{\xi^2-2}~ t)}{|\xi|\sqrt{\xi^2-2}} \xi^2 \hat\varphi_0 (\xi) }_{L^2(|\xi|> \sqrt{2})} ~ \leq  \max\{1, |t|\}  \norm{|\xi|^s \hat\varphi_0(\xi)}_{L^2(|\xi|> \sqrt{2})},
\end{aligned}
\ee
and

\be\label{Estimate_16}
\begin{aligned}
& \norm{|\xi|^s \frac{\sin(|\xi|\sqrt{\xi^2-2} ~ t) }{|\xi| \sqrt{\xi^2-2}} (2- \xi^2) \hat\phi_0(\xi) }_{L^2(|\xi| > \sqrt{2})}  ~ \leq   \max\{1, |t|\} \norm{|\xi|^s \hat\phi_0(\xi)}_{L^2(|\xi| > \sqrt{2})} .
\end{aligned}
\ee
\end{lem}

\begin{proof}
The proof of \eqref{Estimate_14} is direct. Let us show \eqref{Estimate_15}. We have
 \[
 \begin{aligned}
  \norm{ |\xi|^s\frac{ \sin(|\xi| \sqrt{\xi^2-2}~ t)}{|\xi|\sqrt{\xi^2-2}} \xi^2 \hat\varphi_0(\xi) }_{L^2(|\xi|> \sqrt{2})}  \lesssim & ~ \norm{ |\xi|^s\frac{ \sin(|\xi| \sqrt{\xi^2-2}~ t)}{|\xi|\sqrt{\xi^2-2}} \xi^2 \hat\varphi_0(\xi) }_{L^2(\sqrt{2} < |\xi|< 2)} \\
 &    \quad ~ + \norm{ |\xi|^s\frac{ \sin(|\xi| \sqrt{\xi^2-2}~ t)}{|\xi|\sqrt{\xi^2-2}} \xi^2 \hat\varphi_0 (\xi)}_{L^2(|\xi|> 2)}\\
&   \lesssim ~ |t|^{} \norm{|\xi|^s \hat\varphi_0(\xi) }_{L^2(\sqrt{2} < |\xi|< 2)} +\norm{|\xi|^s \hat\varphi_0(\xi) }_{L^2( |\xi|> 2)}\\
& \lesssim  ~ \max\{1, |t|\}   \norm{|\xi|^s \hat\varphi_0(\xi) }_{L^2( |\xi|> \sqrt{2})}.
 \end{aligned}
 \]
 Finally, \eqref{Estimate_16} holds by following similar steps.
\end{proof}

Finally, we consider the inhomogeneous estimates.

\begin{lem}
Assume $s,t\geq 0$. Under the representation \eqref{phi_high} and \eqref{varphi_high}, we have
\be\label{Estimate_24}
\begin{aligned}
 \norm{|\xi|^s \int_0^t \cos(|\xi| \sqrt{\xi^2-2} (t-\sigma))\hat g(\sigma,\xi) d\sigma }_{L^2(|\xi|> \sqrt{2})}  \leq  \int_0^t  \norm{|\xi|^s \hat g(\sigma,\xi) }_{L^2(|\xi|> \sqrt{2})} d\sigma,
\end{aligned}
\ee
\be\label{Estimate_25}
\begin{aligned}
& \norm{|\xi|^s\int_0^t  \frac{ \sin(|\xi| \sqrt{\xi^2-2} (t-\sigma))}{|\xi|\sqrt{\xi^2-2}} \xi^2 \hat f(\sigma,\xi)d\sigma}_{L^2(|\xi|> \sqrt{2})} \\
&\qquad\quad ~ \lesssim  \int_0^t  \max\{ 1,|t-\sigma|\}\norm{|\xi|^s\hat f(\sigma,\xi)}_{L^2(|\xi| > \sqrt{2})}d\sigma,
\end{aligned}
\ee
and
\be\label{Estimate_26}
\begin{aligned}
& \norm{ |\xi|^s\int_0^t  \frac{\sin(|\xi|\sqrt{\xi^2-2} (t-\sigma)) }{|\xi| \sqrt{\xi^2-2}}(2-  \xi^2) \hat g(\sigma,\xi)d\sigma}_{L^2(|\xi|> \sqrt{2})}  \\
&\qquad\quad ~   \lesssim  \int_0^t   \max\{ 1,|t-\sigma|\} \norm{|\xi|^s\hat g(\sigma,\xi)}_{L^2(|\xi|> \sqrt{2})}d\sigma.
\end{aligned}
\ee
\end{lem}
\begin{proof}
Estimate \eqref{Estimate_24} is direct. Estimates \eqref{Estimate_25} and \eqref{Estimate_26} hold following the steps in the proof of estimate \eqref{Estimate_15}.
\end{proof}

Now we can state a corresponding result (as in Corollary \ref{Fourier_low}), for the case of large frequencies. This result will be useful in next section.

\begin{cor}\label{Fourier_high}
Assume $s,t\geq 0$ and consider the Duhamel representation for $\phi$ and $\varphi$ given in \eqref{Duhamel}-\eqref{phi_high}-\eqref{varphi_high}. Then we have

\be\label{Hs_norm_phi_high}
\begin{aligned}
& \| |\xi|^s\mathcal F[ \Phi [f,g ; \phi_0,\varphi_0]] (t,\xi)\|_{L^2(|\xi|> \sqrt{2})} \\
 & \quad  \lesssim_s   \norm{|\xi|^s \hat\phi_0(\xi)}_{L^2(|\xi|> \sqrt{2})}  ~ + \max\{1, |t|\} \norm{|\xi|^s \hat\varphi_0(\xi)}_{L^2(|\xi|> \sqrt{2})} \\
& \qquad \qquad ~ + \int_0^t   \norm{ |\xi|^s\hat g(\sigma,\xi) }_{L^2(|\xi|> \sqrt{2})} d\sigma \\
& \qquad\qquad  ~ + \int_0^t  \max\{1, |t-\sigma|\} \norm{|\xi|^s \hat f(\sigma,\xi)}_{L^2(|\xi| > \sqrt{2})}d\sigma.
\end{aligned}
\ee
Similarly,

\be\label{Hs_norm_varphi_high}
\begin{aligned}
& \| |\xi|^s \mathcal F[ \Psi [f,g ; \phi_0,\varphi_0]] (t,\xi)\|_{L^2(|\xi| > \sqrt{2})} \\
 & \quad \lesssim_s ~  \norm{|\xi|^s\hat\varphi_0(\xi)}_{L^2(|\xi| > \sqrt{2})} + \max\{1, |t|\} \norm{|\xi|^s\hat\varphi_0(\xi)}_{L^2(|\xi|> \sqrt{2})} \\
& \qquad \qquad ~ +\int_0^t    \norm{|\xi|^s \hat f(\sigma,\xi) }_{L^2(|\xi|> \sqrt{2})} d\sigma \\
& \qquad \qquad ~ +\int_0^t  \max\{1, |t-\sigma|\}  \norm{|\xi|^s\hat g(\sigma,\xi)}_{L^2(|\xi| > \sqrt{2})}d\sigma.
\end{aligned}
\ee

\end{cor}

\begin{rem}\label{Contraction_estimates}
Note finally that the following continuity estimates are direct from the previous estimates: for low frequencies,
\be\label{Contraction_phi_low}
\begin{aligned}
& \| |\xi|^s\mathcal F[ \Phi [f_1,g_1 ; \phi_0,\varphi_0]  -\Phi [f_2,g_2 ; \phi_0,\varphi_0]  ] (t,\xi)\|_{L^2(|\xi|\leq \sqrt{2})} \\
 & \quad  \lesssim_s  \int_0^t    \cosh(t-\sigma) \norm{ (\hat g_1 -\hat g_2)(\sigma,\xi) }_{L^2(|\xi|\leq \sqrt{2})} d\sigma \\
& \qquad\qquad  ~ + \int_0^t \sinh (t-\sigma)\norm{(\hat f_1 -\hat f_2)(\sigma,\xi)}_{L^2(|\xi|\leq \sqrt{2})}d\sigma,
\end{aligned}
\ee
while for large frequencies,
\be\label{Contraction_phi_high}
\begin{aligned}
& \| |\xi|^s\mathcal F[ \Phi [f_1,g_1 ; \phi_0,\varphi_0]  - \Phi [f_2,g_2 ; \phi_0,\varphi_0] ]  (t,\xi)\|_{L^2(|\xi|> \sqrt{2})} \\
& \qquad \quad ~ \lesssim \int_0^t   \norm{ |\xi|^s(\hat g_1-\hat g_2)(\sigma,\xi) }_{L^2(|\xi|> \sqrt{2})} d\sigma \\
& \qquad\qquad  ~ + \int_0^t  \max\{1, |t-\sigma |\} \norm{|\xi|^s (\hat f_1 -\hat f_2)(\sigma,\xi)}_{L^2(|\xi| > \sqrt{2})}d\sigma.
\end{aligned}
\ee
and a similar estimate is satisfied by $\Psi$ interchanging the role of $f$ and $g$.
\end{rem}
\medskip

\section{Proof of Theorem \ref{MT}}\label{3}

\medskip

\subsection{Statement of the result, and first estimates} We start with the definition of solution for \eqref{mNLS}.

\begin{defn}[Solution in the Duhamel sense]
Fix $T>0$, and $s\geq 0$. We say that $w\in C([0,T],H^s(\R))$ is a strong solution to \eqref{mNLS} with initial datum $w_0\in H^s(\R)$ if $w$ satisfies the equation
\be\label{Duhamel_2}
\begin{aligned}
i\partial_t  w + \partial_x^2  w +2\re w =&~ G, \\
G =&~ G[w]:= -[2|w|^2 + w^2 + |w|^2w],
\end{aligned}
\ee
in the integral sense. More precisely, under the decomposition \eqref{deco_u0}-\eqref{deco_F},  $w=\phi + i\varphi$ satisfies the Duhamel representation \eqref{Duhamel}, with $f+ig=G$.
\end{defn}

In order to show the existence of a unique solution of the form \eqref{Duhamel_2} to \eqref{mNLS}, we will use a standard contraction principle in the Sobolev space $H^s$. We will show the following detailed statement:

\begin{thm}\label{MT1}
For any $T>0$, there exists $\delta =\delta (T)>0$ such that for any $w_0 \in H^s(\R)$, $s>\frac12$ and $\|w_0\|_{H^s} <\delta$,  there exists a unique solution to \eqref{mNLS} $w\in   C([0,T]; H^s(\R))$.  A corresponding result holds for any $w_0 \in H^s(\R)$ with no size restriction, provided $T>0$ is chosen small enough. Finally, we have the following alternative: if $I$ is the maximal interval of existence of $w$, and $\sup I<+\infty$, then $\lim_{t\uparrow \sup I } \|w(t)\|_{H^s} =+\infty.$ 
\end{thm}

\begin{rem}\label{4.10.3}
It will be clear from the proof that Theorem \ref{MT1} holds for \eqref{NLS} with any $p>1$, provided the regularity space $H^s$, $s>\frac12$ is chosen appropriately, depending on the cases $p$ even, $p$ odd, or $p$ general positive non-integer real number. In each of these three cases, the nonlinearity $|u|^{p-1}u$ has different smoothness, which imply different (and limit) choices for the regularity space $H^s$ where the iteration procedure is performed. See \cite[Remark 4.10.3]{Cazenave} for a detailed account of this fact.
\end{rem}

For the proof of Theorem \ref{MT1} we will need the following preliminary results.

\begin{lem}
Consider $G=G[w]$ as in \eqref{Duhamel_2}. Assume $w\in H^s$, $s>\frac12$. Then $G$ is also in $H^s$, and 
\be\label{estimate_F}
\| G\|_{H^s} \lesssim_s   \|w\|_{H^s}^2 +  \|w\|_{H^s}^3,
\ee
\end{lem} 

\begin{proof}
Direct from the Leibnitz rule of fractional derivatives \cite[Thm. 3.5]{LP} and the fact that $H^s$, $s>\frac12$ is a multiplicative algebra.
\end{proof}

\begin{lem}
Consider $w_1,w_2\in H^s$, $s>\frac12$, and $G_1=G[w_1]$, $G_2=G[w_2]$ as in \eqref{Duhamel_2}. Then we have 
\be\label{estimate_FmG}
\| G_1-G_2 \|_{H^s} \lesssim_s   (\|w_1\|_{H^s} +\|w_2\|_{H^s} +\|w_1\|_{H^s}^2 +\|w_2\|_{H^s}^2 )  \|w_1-w_2\|_{H^s}.
\ee
\end{lem} 

\begin{proof}
Direct from the definition of $G[w]$ in \eqref{Duhamel_2}.
\end{proof}

Now, the remaining part of the proof is standard, see e.g. \cite[Section 4.10]{Cazenave}. Let us consider some parameters $M,T,\delta>0$, and $s>\frac12$, and assume
\be\label{IC_delta}
 \|w_0\|_{H^s} <\delta.
\ee
Note that $\delta$ is not necessarily small. Assume now $M\geq \delta.$ In what follows, we define the Banach space
\be\label{Ball}
\mathcal B(T,M):= \Big\{ w\in C([0,T];H^s) ~ : ~   \sup_{t\in [0,T]} \| w(t)\|_{H^s} \leq M \Big\},
\ee
endowed with the standard $L^\infty_t H^s_x$ norm.

\begin{lem}
Let $w_0$ be as in \eqref{IC_delta}, $w\in \mathcal B(T,M)$, and $G$ as in \eqref{Duhamel_2}. Assume that $w$ and $G$ satisfy the decomposition \eqref{deco_u0}-\eqref{deco_F}, and consider the operators $\Phi$ and $\Psi$ as in Lemma \ref{Fourier_Linear}. Then both $\Phi$ and $\Psi$ are well-defined and one has the estimates
\be\label{Estimate_Phi}
\begin{aligned}
 \|\Phi [f,g ; \phi_0,\varphi_0] (t)\|_{H^s}   \lesssim_s &~  \delta (\cosh T  +   \sinh T) \\
&  ~ +(M^2+M^3) (\sinh T + \cosh T-1)\\
& ~ + \max\{1, T \}  \delta + T (M^2+M^3) ~ + (M^2+M^3)  \max\{1, T\} ~T .
\end{aligned}
\ee
and
\be\label{Estimate_Psi}
\begin{aligned}
 \|\Psi [f,g ; \phi_0,\varphi_0] (t)\|_{H^s}   \lesssim_s &~  \delta (\cosh T  +   \sinh T) \\
&  ~ +(M^2+M^3) (\sinh T + \cosh T-1)\\
& ~ + \max\{1, T \}  \delta  + T (M^2+M^3)  ~ + (M^2+M^3)  \max\{1, T\} ~T .
\end{aligned}
\ee
\end{lem}

\begin{proof}
Let us show \eqref{Estimate_Phi}. Combining \eqref{Hs_norm_phi_low} and \eqref{Hs_norm_phi_high} we have (recall that $\langle \xi \rangle := \sqrt{1+|\xi|^2}$)
\be\label{Auxiliar_0}
\begin{aligned}
 \|\Phi [f,g ; \phi_0,\varphi_0] (t)\|_{H^s}   \lesssim_s &~  \cosh t  \norm{\hat\phi_0(\xi)}_{L^2(|\xi|\leq \sqrt{2})} + \norm{ \langle\xi \rangle^s \hat\phi_0(\xi)}_{L^2(|\xi|> \sqrt{2})} \\
&  ~ + \sinh t  \norm{\hat\varphi_0(\xi)}_{L^2(|\xi|\leq \sqrt{2})} \\
&  ~ + \int_0^t    \cosh(t-\sigma) \norm{ \hat g(\sigma,\xi) }_{L^2(|\xi|\leq \sqrt{2})} d\sigma \\
& ~ + \int_0^t \sinh (t-\sigma)\norm{\hat f(\sigma,\xi)}_{L^2(|\xi|\leq \sqrt{2})}d\sigma\\
& ~ + \max\{1, |t|\}  \norm{\langle \xi \rangle^s \hat\varphi_0(\xi)}_{L^2(|\xi|> \sqrt{2})} \\
&  ~ + \int_0^t   \norm{ \langle \xi \rangle^s\hat g(\sigma,\xi) }_{L^2(|\xi|> \sqrt{2})} d\sigma \\
&  ~ + \int_0^t  \max\{1, |t-\sigma|\} \norm{\langle \xi \rangle^s \hat f(\sigma,\xi)}_{L^2(|\xi| > \sqrt{2})}d\sigma.
\end{aligned}
\ee
From the definition of solution, \eqref{estimate_F}, \eqref{IC_delta} and \eqref{Ball},
\[
\begin{aligned}
 \|\Phi [f,g ; \phi_0,\varphi_0] (t)\|_{H^s}   \lesssim_s &~   \delta \cosh t +  \delta \sinh t \\
&  ~ +(M^2+M^3) \int_0^t   [ \cosh(t-\sigma) +\sinh (t-\sigma)] d\sigma \\
& ~ + \max\{1, |t|\}  \delta +  |t| (M^2+M^3)\\
&  ~ + (M^2+M^3) \int_0^t  \max\{1, |t-\sigma|\} d\sigma.
\end{aligned}
\]
A further simplification gives
\[
\begin{aligned}
 \|\Phi [f,g ; \phi_0,\varphi_0] (t)\|_{H^s}   \lesssim_s &~   \delta \cosh T + \delta  \sinh T \\
&  ~ +(M^2+M^3) (\sinh T + \cosh T-1)\\
& ~ + \max\{1, T \}  \delta + T (M^2+M^3) ~ + (M^2+M^3)  \max\{1, T\} ~T ,
\end{aligned}
\]
as required. The proof of \eqref{Estimate_Psi} can be established in a completely  similar fashion, using \eqref{Hs_norm_varphi_low} and \eqref{Hs_norm_varphi_high}. 
\end{proof}

\begin{lem}[Contraction]\label{Contraction}
Let $w_0$ be as in \eqref{IC_delta}, $w_1,w_2\in \mathcal B(T,M)$, and $G_1=G[w_1]$, $G_2=G[w_2]$ as in \eqref{Duhamel_2}. Consider the natural decomposition $w_1=f_1+ig_2$ and $w_2=f_2 + ig_2$. Then one has the continuity estimates
\be\label{Contra_11}
\begin{aligned}
& \sup_{t\in [0,T]}\|  \Phi [f_1,g_1 ; \phi_0,\varphi_0] (t) -\Phi [f_2,g_2 ; \phi_0,\varphi_0] (t) \|_{H^s} \\
 & \quad  \lesssim_s (M+M^2)  \big( \cosh T-1 +T +\max\{1,T\} T \big) \sup_{t\in [0,T]}  \|(w_1-w_2)(t)\|_{H^s},
\end{aligned}
\ee
and
\be\label{Contra_22}
\begin{aligned}
& \sup_{t\in [0,T]}\|  \Psi [f_1,g_1 ; \phi_0,\varphi_0] (t) -\Psi [f_2,g_2 ; \phi_0,\varphi_0] (t) \|_{H^s} \\
 & \quad  \lesssim_s (M+M^2)  \big( \cosh T-1 +T +\max\{1,T\} T \big) \sup_{t\in [0,T]}  \|(w_1-w_2)(t)\|_{H^s} .
\end{aligned}
\ee
\end{lem}

\begin{proof}
We prove \eqref{Contra_11}. Using Remark \ref{Contraction_estimates}, and more precisely \eqref{Contraction_phi_low}-\eqref{Contraction_phi_high},
\[
\begin{aligned}
&\|  \Phi [f_1,g_1 ; \phi_0,\varphi_0] (t) -\Phi [f_2,g_2 ; \phi_0,\varphi_0] (t) \|_{H^s} \\
 & \quad  \lesssim_s  \int_0^t    \cosh(t-\sigma) \norm{ (\hat g_1 -\hat g_2)(\sigma,\xi) }_{L^2(|\xi|\leq \sqrt{2})} d\sigma \\
& \qquad\qquad  ~ + \int_0^t \sinh (t-\sigma)\norm{(\hat f_1 -\hat f_2)(\sigma,\xi)}_{L^2(|\xi|\leq \sqrt{2})}d\sigma\\
& \qquad \quad ~ + \int_0^t   \norm{ \langle \xi \rangle^s(\hat g_1-\hat g_2)(\sigma,\xi) }_{L^2(|\xi|> \sqrt{2})} d\sigma \\
& \qquad\qquad  ~ + \int_0^t  \max\{1, |t-\sigma|\} \norm{\langle \xi \rangle^s (\hat f_1 -\hat f_2)(\sigma,\xi)}_{L^2(|\xi| > \sqrt{2})}d\sigma.
\end{aligned}
\]
Using now \eqref{estimate_FmG}, we have
\[
\begin{aligned}
&\|  \Phi [f_1,g_1 ; \phi_0,\varphi_0] (t) -\Phi [f_2,g_2 ; \phi_0,\varphi_0] (t) \|_{H^s} \\
 & \quad  \lesssim_s (M+M^2) \int_0^t    \cosh(t-\sigma)   \|(w_1-w_2)(\sigma)\|_{H^s} d\sigma \\
& \qquad \quad ~ +(M+M^2) \int_0^t    \|(w_1-w_2)(\sigma)\|_{H^s}d\sigma \\
& \qquad\qquad  ~ + (M+M^2) \int_0^t  \max\{1, |t-\sigma|\}  \|(w_1-w_2)(\sigma)\|_{H^s}d\sigma.
\end{aligned}
\]
Therefore,
\[
\begin{aligned}
& \sup_{t\in [0,T]}\|  \Phi [f_1,g_1 ; \phi_0,\varphi_0] (t) -\Phi [f_2,g_2 ; \phi_0,\varphi_0] (t) \|_{H^s} \\
 & \quad  \lesssim_s (M+M^2)  \big( \cosh T-1 +T +\max\{1,T\} T \big) \sup_{t\in [0,T]}  \|(w_1-w_2)(t)\|_{H^s} .
\end{aligned}
\]
The proof for $\Psi$ is very similar.
\end{proof}

\subsection{Proof of Theorem \ref{MT1}} First we prove the small time local well-posedness. Let us fix $\delta>0$, and $M=C\delta$, $C>0$ large depending on $s$.  Consider \eqref{Estimate_Phi}. By taking $T>0$ small enough, we have
\[
 \|\Phi [f,g ; \phi_0,\varphi_0] (t)\|_{H^s}   \leq \frac12 M,
\]
and a similar estimate holds for $\Psi$ in \eqref{Estimate_Psi}. The contraction character follows immediately from Lemma \ref{Contraction}, choosing $T$ smaller if necessary. The contraction principle allows to conclude.

\medskip

Now we prove small data local well-posedness. Let us fix $T>0$, and $M=C(T,s)\delta$, with $C$ fixed.  Consider \eqref{Estimate_Phi}. By taking $\delta>0$ small enough, we have
\[
 \|\Phi [f,g ; \phi_0,\varphi_0] (t)\|_{H^s}   \leq \frac12 M,
\]
and a similar estimate holds for $\Psi$ in \eqref{Estimate_Psi}. The contraction character follows immediately from Lemma \ref{Contraction}, choosing $\delta>0$ smaller if necessary. The fixed point principle allows to conclude. This finishes the proof of Theorem \ref{MT1}, except for the blow-up alternative.

\medskip

Now we prove the blow-up alternative. If $\limsup_{t\to \sup I} \|w(t)\|_{H^s}<+\infty $, then $w(t)$ is bounded in $L^\infty_{t,x}$ and from \eqref{Auxiliar_0}, we have, after Gronwall's inequality,
\[
\|\re w(t)\|_{H^s} \lesssim e^{Ct},
\] 
therefore $\re w(t)$ is bounded in $H^s$. A similar conclusion follows for the $H^s$ norm of $\ima w(t)$, and therefore $I$ is not maximal. The proof of Theorem \ref{MT1} is complete.

\begin{rem}[About global existence]\label{GE} Assume now $w_0\in H^1(\R)$. From the previous subsection, we have a local solution in $H^1$. 
We will review now some almost conservation laws. It is known that \eqref{NLS} for $p=3$ possesses the following formally conserved quantities:
\be\label{Mass}
M[u]: =\int (|u|^2 -1), \qquad \hbox{(Mass)}
\ee
and 
\be\label{Energy}
E[u]:=\frac 12 \int |\partial_x u|^2 - \frac14 \int (|u|^2-1)^2, \qquad \hbox{(Energy)}
\ee
which induce the following conserved identities for \eqref{mNLS}:
\be\label{Mass_w}
\int |w|^2 + 2\re w = \hbox{conserved},
\ee
and 
\be\label{Energy_w}
 \int |\partial_x w|^2 - \frac12 \int (|w|^2 + 2\re w)^2 =\hbox{conserved}.
\ee
Note that \eqref{Mass_w} is not well-defined for solutions in $H^1$ only. The second quantity \eqref{Energy_w} is conserved as long as the solution $w(t)$ remains in $H^1(\R)$, as a standard density argument involving smooth and rapidly decaying solutions shows. However, unlike as in the defocusing case \cite{Zhidkov, Gallo, Gerard}, this energy does not give a priori control of the dynamics. A third conserved quantity comes from the momentum law:
\be\label{Momentum_w}
\ima \int \overline{w} \partial_x w = \hbox{conserved}.
\ee
We conclude that the fact that \eqref{mNLS} has no conservation of mass reveals that it is not clear whether or not there is a reasonable GWP result, except possibly for small data.
%
\end{rem}

\begin{rem}
Another measure of the ``unstable'' character of \eqref{mNLS} (even present in the Schwartz class) is given by the following formal computation.  Consider the linear equation \eqref{LS} for simplicity, and assume that $w(t) \in L^1(\R)$ in \eqref{LS}. Then we have
\[
\frac d{dt} \int \re w(t) = 0, \qquad \frac d{dt} \int \ima w(t) = 2\int \re w(t),  
\]  
which implies that the zero Fourier mode of $w(t)$ grows linearly in time, {\bf unless} $\displaystyle{\int \re w(0) =0}$. It is not clear to the author if a suitable bounded dynamics (for the nonlinear problem) can be obtained by assuming (and propagating) this condition on the initial datum. See e.g. \cite{Goodman} for a similar study in the case of the periodic Kuramoto-Sivashinsky equation. The reader may compare this strong instability result with the weak \emph{logarithmic} grow appearing in a similar NLS equation coming from the study of the binormal flow, see Banica and Vega \cite{BV}. 
\end{rem}

\medskip

\section{Some applications to breather solutions}\label{4}

In this Section we prove Theorems \ref{Insta_Per} and \ref{Insta_KM}. Recall that we have now the particular case $p=3$, which is integrable.

\subsection{Proof of Theorem \ref{Insta_Per}}
This proof is not difficult, and it is based in the notion of  \emph{asymptotic stability}, see e.g. \cite[p. 2]{KMM}. Fix $s>\frac12$. Let us assume that the Peregrine breather $P$ in \eqref{Peregrine} is orbitally stable, as in \eqref{Stability}. Write
\be\label{Q_def}
P(t,x)= e^{it} (1+ Q(t,x)),\qquad Q(t,x):= ~ -\frac{4(1+2it)}{1+ 4t^2 +2x^2}.
\ee
Now consider, as a perturbation of the Peregrine breather, the Stokes wave \eqref{Stokes}.  Indeed, we have (see \eqref{Q_def}),
\[
\lim_{t\to +\infty} \|  e^{it} -P(t) \|_{H^s}= \lim_{t\to +\infty} \| Q(t) \|_{H^s} =0.
\]
Therefore, we have two modulationally unstable solutions to \eqref{NLS} that converge to the same profile as $t\to +\infty$. This fact contradicts the orbital stability, since for $x_0(t)\in \R$ given in \eqref{Stability},
\[
0< c_0 := \| Q(0,x-x_0(0)) \|_{H^s}
\]
is a fixed number, but if $t_0=T$ is taken large enough, $ \| Q(T) \|_{H^s}$ can be made arbitrarily small. This proves Theorem \ref{Insta_Per}.

\subsection{Proof of Theorem \ref{Insta_KM}} The proof is similar to the one in the previous case. Now we consider the Kuznetsov-Ma breather $B$ in \eqref{KM}. Assume that this breather is orbitally stable. Consider as a perturbation the explicit solution $\tilde B$ to \eqref{NLS} that describes the nonlinear addition of a Kuznetsov-Ma breather, and a Peregrine breather, both well-decoupled. This solution can be easily obtained from \cite[Appendix B]{Akhmediev3} after performing the standard changes and the Peregrine's limit procedure. It is also seen (\cite[eqn. (17)]{Akhmediev3}) that 
\[
\tilde B(t,x) = B(t,x) + \tilde P(t,x),
\]
where
\[
\lim_{t\to +\infty} \|  \tilde P(t) \|_{H^s} =0.
\]
However, for $x_0(t)\in \R$ given in \eqref{Stability}, we have $\| \tilde B(0) - B(0,x-x_0(0)) \|_{H^s} =c_1>0$, a fixed nonzero constant. However, for $t_0=T$ arbitrarily large, $ \| \tilde B(T) - B(T) \|_{H^s} $ can be made as small as possible, a contradiction. 

\begin{rem}
As a conclusion, the two preceding proofs reveal that any suitable nonlinear wave appearing as a product of the modulational instability of the Stokes wave in the equation, must be orbitally unstable, and no asymptotic stability should hold.    
\end{rem}


\begin{rem}
For further purposes, we compute the mass and energy \eqref{Mass}-\eqref{Energy} of the Peregrine \eqref{Peregrine} and Kuznetsov-Ma \eqref{KM} breathers. We have
\[
M[P]= E[P]=0,
\]
(however, the $L^2$-norm of $Q(t)$ is never zero, but converges to zero as $t\to +\infty$), and
\[
M[B] = 4\beta, \quad E[B] =-\frac 43\bt^3.
\]
Note that $P$ has same energy and mass as the Stokes wave solution (the nonzero background), a property not satisfied by the standard soliton on zero background. Also, compare the mass and energy of the Kuznetsov-Ma breather with the ones obtained in \cite{AM} for the mKdV breather.
\end{rem}

\medskip

\appendix

\section{Proof of Lemma \ref{Fourier_Linear}}\label{A}

\medskip

Let us consider once again the linear equation in \eqref{Linear}:
\[
i\partial_t w +\partial_x^2w +2\re w =G, \qquad w(t=0)=w_0.
\]
Using the decomposition in \eqref{deco_u0}-\eqref{deco_F}, we have
\[
i(\partial_t \phi + i \partial_t \varphi) +\partial_x^2 (\phi+ i\varphi) +2\phi =f+ ig.
\]
Hence,
\be\label{base_eqns}
\begin{aligned}
\partial_t \phi +\partial_x^2 \varphi = & ~g,\\
-\partial_t \varphi +\partial_x^2 \phi +2\phi = & ~f.
\end{aligned}
\ee
Therefore
\[
\partial_t^2 \phi = -\partial_t \partial_x^2 \varphi + \partial_t g = -\partial_x^2( \partial_x^2 \phi +2\phi -f) +\partial_t g, 
\]
so that
\be\label{eq_phi}
\partial_t^2 \phi + \partial_x^2(\partial_x^2 +2) \phi  = \partial_x^2 f+ \partial_t g.
\ee
Similarly,
\[
\partial_t^2 \varphi =  (\partial_x^2 +2)\partial_t \phi - \partial_t f = -\partial_x^2( \partial_x^2 +2)\varphi + ( \partial_x^2 +2)g  -\partial_t f, 
\]
so that
\be\label{eq_varphi}
\partial_t^2 \varphi +\partial_x^2( \partial_x^2 +2)\varphi =( \partial_x^2 +2)g  -\partial_t f.
\ee
Writing \eqref{eq_phi} in Fourier variables, we have
\be\label{AA}
\partial_t^2 \hat \phi +\xi^2 (\xi^2 -2) \hat\phi = -\xi^2 \hat f + \partial_t \hat g. 
\ee
Now we have to consider two different cases. 

\medskip

\noindent
{\bf Case $|\xi|\leq \sqrt{2}$.} In this case $\xi^2 (\xi^2 -2) \leq 0$. The solution to this ODE is given by the formula
\[
\begin{aligned}
 \hat \phi (t,\xi) =&~ \cosh(|\xi|\sqrt{2-\xi^2} ~ t) \hat\phi_0(\xi) + \frac{\sinh(|\xi|\sqrt{2-\xi^2} ~ t) }{|\xi| \sqrt{2-\xi^2}}  \hat\phi_1(\xi)\\
 &~ +\int_0^t  \frac{\sinh(|\xi|\sqrt{2-\xi^2} (t-\sigma)) }{|\xi| \sqrt{2-\xi^2}}  [-\xi^2 \hat f(\sigma,\xi) + \partial_t \hat g(\sigma,\xi) ]d\sigma.
\end{aligned}
\]
Here, $ \hat\phi_1(\xi) := \mathcal F[ \partial_t \phi (0,x)](\xi)$. We have from \eqref{base_eqns},
\[
 \mathcal F[ \partial_t \phi (0,x)](\xi) = \mathcal F[ -\partial_x^2 \varphi(0,x) + g(0,x) ](\xi) = \xi^2 \hat\varphi_0(\xi) +\hat g(0,\xi).
\]
On the other hand,
\[
\begin{aligned}
\int_0^t  \frac{\sinh(|\xi|\sqrt{2-\xi^2} (t-\sigma)) }{|\xi| \sqrt{2-\xi^2}}   \partial_t \hat g(\sigma,\xi) d\sigma =& ~  - \frac{\sinh(|\xi|\sqrt{2-\xi^2} ~ t) }{|\xi| \sqrt{2-\xi^2}}\hat g(0,\xi)  \\
& ~ + \int_0^t  \cosh(|\xi|\sqrt{2-\xi^2} (t-\sigma)) \hat g(\sigma,\xi) d\sigma.
\end{aligned}
\]
Consequently,
\[
\begin{aligned}
 \hat \phi (t,\xi) =&~ \cosh(|\xi|\sqrt{2-\xi^2} ~ t) \hat\phi_0(\xi) + \frac{\sinh(|\xi|\sqrt{2-\xi^2} ~ t) }{|\xi| \sqrt{2-\xi^2}}  \xi^2 \hat\varphi_0(\xi) \\
 & ~+ \int_0^t  \cosh(|\xi|\sqrt{2-\xi^2} (t-\sigma)) \hat g(\sigma,\xi) d\sigma\\
 &~ -\int_0^t  \frac{\sinh(|\xi|\sqrt{2-\xi^2} (t-\sigma)) }{|\xi| \sqrt{2-\xi^2}}  \xi^2 \hat f(\sigma,\xi)d\sigma,
\end{aligned}
\]
which proves \eqref{phi_low}.

\medskip

\noindent
{\bf Case $|\xi| > \sqrt{2}$.} Now we have $\xi^2 (\xi^2 -2)>0$ and
\[
\begin{aligned}
 \hat \phi (t,\xi) =&~ \cos(|\xi|\sqrt{\xi^2-2} ~ t) \hat\phi_0(\xi) + \frac{\sin(|\xi|\sqrt{\xi^2-2} ~ t) }{|\xi| \sqrt{\xi^2-2}}  \hat\phi_1(\xi)\\
 &~ +\int_0^t  \frac{\sin(|\xi|\sqrt{\xi^2-2} (t-\sigma)) }{|\xi| \sqrt{\xi^2-2}}  [-\xi^2 \hat f(\sigma,\xi) + \partial_t \hat g(\sigma,\xi) ]d\sigma.
\end{aligned}
\]
Here, as in the previous case, $ \hat\phi_1(\xi) := \mathcal F[ \partial_t \phi (0,x)](\xi)= \xi^2 \hat\varphi_0(\xi) +\hat g(0,\xi)$ (see \eqref{base_eqns}).  Finally, 
\[
\begin{aligned}
\int_0^t  \frac{\sin(|\xi|\sqrt{\xi^2-2} (t-\sigma)) }{|\xi| \sqrt{\xi^2-2}}   \partial_t \hat g(\sigma,\xi) d\sigma =& ~ -\frac{\sin(|\xi|\sqrt{\xi^2-2} ~ t) }{|\xi| \sqrt{\xi^2-2}} \hat g(0,\xi) \\
& ~ +  \int_0^t  \cos(|\xi|\sqrt{\xi^2-2} (t-\sigma)) \hat g(\sigma,\xi) d\sigma.
\end{aligned}
\]
Consequently, we have the simplified expression
\[
\begin{aligned}
 \hat \phi (t,\xi) =&~ \cos(|\xi|\sqrt{\xi^2-2} ~ t) \hat\phi_0(\xi) + \frac{\sin(|\xi|\sqrt{\xi^2-2} ~ t) }{|\xi| \sqrt{\xi^2-2}}  \xi^2 \hat\varphi_0(\xi) \\
 & ~+ \int_0^t  \cos(|\xi|\sqrt{\xi^2-2} (t-\sigma)) \hat g(\sigma,\xi) d\sigma\\
 &~ -\int_0^t  \frac{\sin(|\xi|\sqrt{\xi^2-2} (t-\sigma)) }{|\xi| \sqrt{\xi^2-2}}  \xi^2 \hat f(\sigma,\xi)d\sigma.
\end{aligned}
\]
which is nothing but \eqref{phi_high}.

\medskip

Now we deal with the Fourier representation for $\varphi$. Using \eqref{eq_varphi},
\[
\partial_t^2 \hat\varphi +\xi^2 (\xi^2 -2)\hat \varphi = (2-\xi^2)\hat g  -\partial_t \hat f.
\]
As in the case of $\phi$, we consider two different cases.

\medskip

\noindent
{\bf Case $|\xi|\leq \sqrt{2}$.} We have
\[
\begin{aligned}
 \hat \varphi (t,\xi) =&~ \cosh(|\xi|\sqrt{2-\xi^2} ~ t) \hat\varphi_0(\xi) + \frac{\sinh(|\xi|\sqrt{2-\xi^2} ~ t) }{|\xi| \sqrt{2-\xi^2}}  \hat\varphi_1(\xi)\\
 &~ +\int_0^t  \frac{\sinh(|\xi|\sqrt{2-\xi^2} (t-\sigma)) }{|\xi| \sqrt{2-\xi^2}}  [ (2-\xi^2)\hat g(\sigma,\xi)  -\partial_t \hat f(\sigma,\xi)  ]d\sigma.
\end{aligned}
\]
Here, $ \hat\varphi_1(\xi) := \mathcal F[ \partial_t \varphi (0,x)](\xi)$. We have from \eqref{base_eqns},
\[
 \mathcal F[ \partial_t \varphi (0,x)](\xi) = \mathcal F[ (\partial_x^2+2) \phi(0,x) - f(0,x) ](\xi) =(2- \xi^2) \hat\phi_0(\xi) -\hat f(0,\xi).
\]
Finally, 
\[
\begin{aligned}
\int_0^t  \frac{\sinh(|\xi|\sqrt{2-\xi^2} (t-\sigma)) }{|\xi| \sqrt{2-\xi^2}}   \partial_t \hat f(\sigma,\xi) d\sigma =& ~- \frac{\sinh(|\xi|\sqrt{2-\xi^2} ~t ) }{|\xi| \sqrt{2-\xi^2}}    \hat f(0,\xi) \\
& ~+   \int_0^t  \cosh(|\xi|\sqrt{2-\xi^2} (t-\sigma)) \hat f(\sigma,\xi) d\sigma.
\end{aligned}
\]
Consequently,
\[
\begin{aligned}
 \hat \varphi (t,\xi) =&~ \cosh(|\xi|\sqrt{2-\xi^2} ~ t) \hat\varphi_0(\xi) + \frac{\sinh(|\xi|\sqrt{2-\xi^2} ~ t) }{|\xi| \sqrt{2-\xi^2}} (2- \xi^2) \hat\phi_0(\xi)  \\
 & ~- \int_0^t  \cosh(|\xi|\sqrt{2-\xi^2} (t-\sigma)) \hat f(s,\xi) d\sigma \\
 &~ + \int_0^t  \frac{\sinh(|\xi|\sqrt{2-\xi^2} (t-\sigma)) }{|\xi| \sqrt{2-\xi^2}}(2-  \xi^2) \hat g(\sigma,\xi)d\sigma,
\end{aligned}
\]
which proves \eqref{varphi_low}.

\medskip

\noindent
{\bf Case $|\xi| > \sqrt{2}$.} We have
\[
\begin{aligned}
 \hat \varphi (t,\xi) =&~ \cos(|\xi|\sqrt{\xi^2-2} ~ t) \hat\varphi_0(\xi) + \frac{\sin(|\xi|\sqrt{\xi^2-2} ~ t) }{|\xi| \sqrt{\xi^2-2}}  [(2- \xi^2) \hat\phi_0(\xi) -\hat f(0,\xi)] \\
 &~ +\int_0^t  \frac{\sin(|\xi|\sqrt{\xi^2-2} (t-\sigma)) }{|\xi| \sqrt{\xi^2-2}}  [(2-\xi^2)\hat g(\sigma,\xi)  -\partial_t \hat f(\sigma,\xi)  ]d\sigma.
\end{aligned}
\]
Finally, 
\[
\begin{aligned}
\int_0^t  \frac{\sin(|\xi|\sqrt{\xi^2-2} (t-\sigma)) }{|\xi| \sqrt{\xi^2-2}}   \partial_t \hat f(\sigma,\xi) d\sigma = &~ - \frac{\sin(|\xi|\sqrt{\xi^2-2} ~t ) }{|\xi| \sqrt{\xi^2-2}}  \hat f(0,\xi) \\
& ~   + \int_0^t  \cos(|\xi|\sqrt{\xi^2-2} (t-\sigma)) \hat f(\sigma,\xi) d\sigma.
\end{aligned}
\]
Consequently, we have
\[
\begin{aligned}
 \hat \varphi (t,\xi) =&~ \cos(|\xi|\sqrt{\xi^2-2} ~ t) \hat\varphi_0(\xi) + \frac{\sin(|\xi|\sqrt{\xi^2-2} ~ t) }{|\xi| \sqrt{\xi^2-2}}  (2- \xi^2) \hat\phi_0(\xi) \\
 & ~- \int_0^t  \cos(|\xi|\sqrt{\xi^2-2} (t-\sigma)) \hat f(\sigma,\xi) d\sigma\\
 &~ + \int_0^t  \frac{\sin(|\xi|\sqrt{\xi^2-2} (t-\sigma)) }{|\xi| \sqrt{\xi^2-2}} (2-\xi^2)\hat g(\sigma,\xi)  d\sigma.
\end{aligned}
\]
which is nothing but \eqref{varphi_high}.

\medskip

\section{Proof of \eqref{ineq_1}}\label{B}

We want to prove, for $t\geq 0$, 
\[
\sup_{|\xi|\leq \sqrt{2} }\frac{\sinh(|\xi| \sqrt{2-\xi^2} t)}{|\xi|\sqrt{2-\xi^2} }\leq \sinh t.
\]
We have
\[
\begin{aligned}
\sup_{|\xi|\leq \sqrt{2} }\frac{\sinh(|\xi| \sqrt{2-\xi^2} ~t)}{|\xi|\sqrt{2-\xi^2} } =&~ \sup_{\xi \in [0, \sqrt{2}] }\frac{\sinh(\xi \sqrt{2-\xi^2}~ t)}{\xi\sqrt{2-\xi^2} } \\
=&~ \max_{u \in [0, 1] }\frac{\sinh(u \,t)}{u }.
\end{aligned}
\]
The maximum value of the last quantity above is attained at the boundary $u=1$, being $\sinh t$. The other extremal point is $u=0$, for which 
\[
\lim_{u\to 0}\frac{\sinh(u\, t)}{u } = t \leq \sinh t, \quad t\geq 0.
\] 
Consequently, 
\[
 \max_{u \in [0, 1] }\frac{\sinh(u\, t)}{u } = \sinh t,
\]
which proves \eqref{ineq_1}.

\providecommand{\bysame}{\leavevmode\hbox to3em{\hrulefill}\thinspace}
\providecommand{\MR}{\relax\ifhmode\unskip\space\fi MR }
\providecommand{\MRhref}[2]{%
  \href{http://www.ams.org/mathscinet-getitem?mr=#1}{#2}
}
\providecommand{\href}[2]{#2}


\begin{thebibliography}{1}

\medskip

\bibitem{Akhmediev2} N. Akhmediev, A. Ankiewicz, and M. Taki, \emph{Waves that appear from nowhere and disappear without a trace}, Phys. Lett. A 373 (2009) 675--678.

\bibitem{Akhmediev1} N. Akhmediev, V. M. Eleonskii, and N. E. Kulagin, \emph{Exact first-order solutions of the nonlinear Schr\"odinger equation}. Theor. Math. Phys. 72, 809--818 (1987).

\bibitem{Akhmediev} N. Akhmediev, and V. I. Korneev, \emph{Modulation instability and periodic solutions of the nonlinear Schr\"odinger equation}. Theor. Math. Phys. 69, 1089--1093 (1986).

\bibitem{Akhmediev3} N. Akhmediev, J. M. Soto-Crespo, and A. Ankiewicz, \emph{Extreme waves that appear from nowhere: On the nature of rogue waves}, Physics Letters A 373 (2009) 2137--2145.

\bibitem{AM} M. A. Alejo, and C. Mu\~noz, \emph{Nonlinear stability of mKdV breathers}, Comm. Math. Phys. (2013), Vol. 324, Issue 1, pp. 233--262.

\bibitem{AM1} M. A. Alejo, and C. Mu\~noz, \emph{Dynamics of complex-valued modified KdV solitons with applications to the stability of breathers}, Anal. and PDE. \textbf{8} (2015), no. 3, 629--674.

\bibitem{AM2} M. A. Alejo, C. Mu\~noz, and J. M. Palacios, \emph{On the Variational Structure of Breather Solutions}, preprint  arXiv:1309.0625.

\bibitem{BV} V. Banica, and L. Vega, \emph{Scattering for 1D cubic NLS and singular vortex dynamics}, J. Eur. Math. Soc. (JEMS) 14 (2012), no. 1, 209--253.

\bibitem{BM} G. Biondini, and D. Mantzavinos, \emph{Long-time asymptotics for the focusing nonlinear Schr\"odinger
equation with nonzero boundary conditions at infinity and asymptotic stage of modulational instability}, preprint arXiv:1512.06095; to appear in Comm. Pure and Applied Math. (2016).

\bibitem{BM1} G. Biondini, and D. Mantzavinos, \emph{Universal nature of the nonlinear stage of modulational instability}, Phys. Rev. Lett. 116 (2016), 043902.

\bibitem{BPSS} J. L. Bona, G. Ponce, J.-C. Saut, and C. Sparber, \emph{Dispersive blow-up for nonlinear Schr\"odinger equations revisited}, J. Math. Pures Appl. 102 (2014), no. 4, 782--811.

\bibitem{Bona_Saut} J. L. Bona, and J.-C. Saut, \emph{Dispersive Blowup of solutions of generalized Korteweg-de Vries equations}, J. Diff. Eqns. 103 no. 1 (1993) 3--57.

\bibitem{Brunetti} M. Brunetti, N. Marchiando, N. Betti, and J. Kasparian, \emph{Nonlinear fast growth of water waves under wind forcing}, Phys. Lett. A 378 (2014) 1025--1030.

\bibitem{Cazenave}  T. Cazenave, \emph{Semilinear Schr\"odinger equations}, Courant Lecture Notes in Mathematics, 10. New York University, Courant Institute of Mathematical Sciences, New York; American Mathematical Society, Providence, RI, 2003. xiv+323 pp. ISBN: 0-8218-3399-5.

\bibitem{CL} T. Cazenave, and P.-L. Lions, \emph{Orbital stability of standing waves for some nonlinear Schr\"odinger equations},  Comm. Math. Phys.  \textbf{85} (1982), no. 4, 549--561.


\bibitem{CW} T. Cazenave, and F. Weissler, \emph{The Cauchy problem for the critical nonlinear Schršdinger equation in $H^s$}, Nonlinear Anal. 14 (1990), no. 10, 807--836.

\bibitem{Dysthe} K. Dysthe, and K. Trulsen, \emph{Note on breather type solutions of the NLS as models for freak waves}, Phys. Scr. Vol. T82, 48--52 (1999). 

\bibitem{Gallo} C. Gallo, \emph{Schr\"odinger group on Zhidkov spaces},  Adv. Differential Equations 9 (2004), no. 5-6, 509--538. 

\bibitem{Gallo2} C. Gallo, \emph{The Cauchy Problem for Defocusing Nonlinear Schr\"odinger Equations with Non-Vanishing Initial Data at Infinity}, Comm. in Partial Diff. Eqns., 33: 729--771 (2008).

\bibitem{Gerard} P. G\'erard, \emph{The Cauchy problem for the Gross-Pitaevskii equation}, Ann. I. H. Poincar\'e -AN 23 (2006) 765--779.

\bibitem{Goodman} Goodman, J., \emph{Stability of the Kuramoto-Sivashinsky and related systems},  Comm. Pure Appl. Math. 47 (1994), no. 3, 293--306.

\bibitem{GV} J. Ginibre, and G. Velo, \emph{On a class of nonlinear Schršdinger equations. I: The Cauchy problem}, J. Funct. Anal. 32 (1979), 1--32. 

\bibitem{GSS} M. Grillakis, J. Shatah, and W. Strauss, \emph{Stability theory of solitary waves in the presence of symmetry. I}.  J. Funct. Anal.  \textbf{74}  (1987),  no. 1, 160--197.

\bibitem{GNT1} S. Gustafson, K. Nakanishi, and T. P. Tsai, \emph{Scattering theory for the Gross-Pitaevskii equation}, Math. Res. Lett. 13 (2006), 273--285. 

\bibitem{GNT2} S. Gustafson, K. Nakanishi, and T. P. Tsai, \emph{Global dispersive solutions for the Gross-Pitaevskii equation in two and three dimensions}, Ann. Henri Poincar\'e 8 (2007), 1303--1331.

\bibitem{Hadamard} J. Hadamard, \emph{Sur les probl\`emes aux d\'eriv\'ees partielles et leur signification physique}, Princeton University Bulletin (1902), pp. 49--52.

\bibitem{KL} Kalantarov, V. K., and Ladyzhenskaya, O. A.,  \emph{The occurrence of collapse for quasilinear equations of parabolic and hyperbolic types}, J. Math. Sci. (1978) 10: 53. doi:10.1007/BF01109723

\bibitem{Kato} T. Kato. \emph{On the Cauchy problem for the (generalized) Korteweg-de Vries equation}. Adv. in Math. Suppl. Stud., Stud. in Appl. Math, 8 (1983), 93--128.

\bibitem{KPV} C. Kenig, G. Ponce, L. Vega. \emph{On the ill-posedness of some canonical dispersive equations}. Duke Math. J, 106 (2001), 617--633.

\bibitem{Khawaja} U. Al Khawaja, H. Bahlouli, M. Asad-uz-zaman, and S.M. Al-Marzoug, \emph{Modulational instability analysis of the Peregrine soliton}, Comm. Nonlinear Sci. Num. Simul., Vol. 19, Issue 8, August 2014, pp. 2706--2714.

\bibitem{Kliber0} B. Kibler, J. Fatome, C. Finot,	G. Millot, F. Dias, G. Genty, N. Akhmediev, and J. M. Dudley, \emph{The Peregrine soliton in nonlinear fibre optics}, Nature Physics 6, 790--795 (2010) doi:10.1038/nphys1740.

\bibitem{Kibler} B. Kibler, J. Fatome, C. Finot, G. Millot, G. Genty, B. Wetzel, N. Akhmediev, F. Dias, and J. M. Dudley, \emph{Observation of Kuznetsov-Ma soliton dynamics in optical fibre}, Sci. Rep. 2, 463; DOI:10.1038/srep00463 (2012).

\bibitem{KH} C. Klein, and M. Haragus, \emph{Numerical study of the stability of the Peregrine breather}, preprint arXiv:1507.06766.

\bibitem{KMM} M. Kowalczyk, Y. Martel, and C. Mu\~noz, \emph{Kink dynamics in the $\phi^4$ model: asymptotic stability for odd perturbations in the energy space}, preprint arXiv:1506.07420 (2015), to appear in J. Amer. Math. Soc.

\bibitem{Kuznetsov} E. Kuznetsov, \emph{Solitons in a parametrically unstable plasma}, Sov. Phys. Dokl. 22, 507--508 (1977).

\bibitem{Leray} J. Leray, \emph{Sur le mouvement d'{}un liquide visqueux emplissant l'{}espace}, Acta Math. 63 (1934), no. 1, 193--248. 

\bibitem{LP}  F. Linares, and G. Ponce, \emph{Introduction to nonlinear dispersive equations}, Second edition. Universitext. Springer, New York, 2015. xiv+301 pp. 

\bibitem{Ma} Y. C. Ma, \emph{The perturbed plane-wave solutions of the cubic Schr\"odinger equation},  Stud. Appl. Math. 60, 43--58 (1979).

\bibitem{Munoz} C. Mu\~noz, \emph{Stability of integrable and nonintegrable structures}, Adv. Differential Equations 19 (2014), no. 9-10, 947--996.

\bibitem{Peregrine} D.H. Peregrine, \emph{Water waves, nonlinear Schr\"odinger equations and their solutions}, J. Austral. Math. Soc. Ser. B 25, 16--43 (1983).

\bibitem{Shrira} V. I. Shrira, and Y. V. Geogjaev, \emph{What make the Peregrine soliton so special as a prototype of freak waves?} J. Eng. Math. 67, 11--22 (2010).

\bibitem{Tsutsumi} Y. Tsutsumi, \emph{$L^2$-solutions for nonlinear Schr\"odinger equations and nonlinear groups}, Funkcial. Ekvac. 30 (1987), no. 1, 115--125.

\bibitem{Weinstein} M.I. Weinstein, \emph{Lyapunov stability of ground states of nonlinear dispersive evolution equations}, Comm. Pure Appl. Math. \textbf{39}, (1986) 51--68.

\bibitem{Youtube} B. Wetzel, Youtube video at \url{https://www.youtube.com/watch?v=LHqYqiYBSE0}.

\bibitem{Zakharov} V. E. Zakharov, A. I. Dyachenko, and A. O. Prokofiev, \emph{Freak waves as nonlinear stage of Stokes wave modulation instability}, Eur. J. Mech. B - Fluids 5, 677--692 (2006).

\bibitem{Zakharov0} V. E. Zakharov, and A. B. Shabat,  \emph{Exact theory of two-dimensional self-focusing and one-dimensional self-modulation of waves in nonlinear media}, JETP, 34 (1): 62--69.

\bibitem{Zhidkov} P. Zhidkov, \emph{The Cauchy problem for a nonlinear Schr\"odinger equation}, Dubna, 1987.

\end{thebibliography}
\end{document}